\documentclass[reqno,onecolumn,oneside]{paper}
\usepackage[english]{babel}
\usepackage{enumerate,amsmath,amsthm,amsfonts,
amssymb,latexsym,mathrsfs,
}
\usepackage[normalem]{ulem}
\usepackage[all]{xy}
\usepackage[sort]{cite}

\usepackage[mathscr]{euscript}
 \let\mathscr\relax
\usepackage[scr]{rsfso}
\usepackage{tikz-cd}

\newtheorem{theorem}{Theorem}[section]
\newtheorem{proposition}[theorem]{Proposition}
\newtheorem{lemma}[theorem]{Lemma}

\theoremstyle{definition}
\newtheorem{definition}[theorem]{Definition}
\newtheorem{exm}[theorem]{Example}

\theoremstyle{remark}

\newcommand{\ax}{\operatorname{Ax}}

\newcommand{\obj}{\operatorname{Obj}}

\newcommand{\id}{\operatorname{id}}

\newcommand{\Mod}{\textrm{-}\mathcal{M}\!\!\:\mathit{od}}

\newcommand{\lang}{\mathcal{L}}
\newcommand{\fml}{\mathit{Fm}_\mathcal{L}}

\newcommand{\fmlp}{\mathit{Fm}_{\cat{L'}}}

\newcommand{\sfm}{\Sigma_{\cat L}}
\newcommand{\Sfm}{\wp\sfm}
\newcommand{\sfmm}{\Sigma_{\cat{L'}}}
\newcommand{\Sfmm}{\wp\sfmm}

\newcommand{\Sfa}{\wp\Sigma_{\cat{L}_1}}
\newcommand{\Sfb}{\wp\Sigma_{\cat{L}_2}}
\newcommand{\QM}{\mathcal{QM}}
\newcommand{\DS}{\mathcal{DS}}

\newcommand{\SL}{\mathcal{SL}}
\newcommand{\Q}{\mathcal{Q}}

\newcommand{\cat}{\mathcal}

\renewcommand{\AA}{\mathbf{A}}

\newcommand{\Eq}{\mathit{Eq}}
\newcommand{\seq}{\mathit{Seq}}

\renewcommand{\th}{\mathit{Th}}

\newcommand{\V}{\operatorname{Var}}
\newcommand{\VV}{\mathbf{V}}

\renewcommand{\phi}{\varphi}
\renewcommand{\theta}{\vartheta}

\newcommand{\g}{\gamma}

\newcommand{\restr}{\upharpoonright}
\newcommand{\la}{\left\langle}
\newcommand{\ra}{\right\rangle}

\newcommand{\under}{\backslash}
\newcommand{\ost}{{}_\ast/}
\newcommand{\lto}{\longrightarrow}

\newcommand{\lmapsto}{\longmapsto}
\newcommand{\ust}{\under_\ast}
\newcommand{\ov}{\overline}

\newcommand{\tensor}{\otimes}

\def\amslatex\slash{{\protect\AmS-\protect\LaTeX}}

\begin{document} 

\title{The category of propositional deductive systems}

\author{Ciro Russo}
\institution{Departamento de Matem\'atica \\ Universidade Federal da Bahia -- BA, Brazil \\ \small{\texttt{ciro.russo@ufba.br}}}
\maketitle
\date{}

\begin{abstract}
We define the category $\QM$ of quantales and their modules and prove the existence of coproducts, and the existence of pushout and amalgamated coproducts under certain conditions. Then we define the non-full subcategory $\DS_0$ of propositional deductive systems, show that it is equivalent to the one of ``real'' propositional logics whose morphisms are interpretations (modulo a language translation, when needed), and prove that the coproduct in $\DS_0$ is precisely the deductive system called ``logical coproduct'' in \cite{ruslu}. Last, we discuss amalgamation in $\DS_0$.
\end{abstract}

\section{Introduction}

The abstract approach to deductive systems has a long history, from the representation by means of closure operators, which dates back at least to Tarski's work (the reader may refer to \cite{woj} for an account of the approach up to the end of the last century), to the most recent proposals involving ordered algebraic structures \cite{cinmor,blojon,galtsi,rusapal,ruslu} or category theory \cite{arnmarpin18,marpin17}.

Problems like equivalence of logics, interpretation between logical systems, algebraizability, and merging different systems have been treated within such abstract frameworks, generating more or less complicated results and constructions.

The representation of deductive systems by means of quantales, quantale modules, and structural closure operators was first introduced by Galatos and Tsinakis in \cite{galtsi}, and proved to be rather effective with respect to various of the aforementioned problems as well as other ones. Indeed, this paper follows previous works of the author in which Galatos and Tsinakis' approach has been developed quite extensively, and successfully applied to an abstract treatment of translations and interpretations among logical systems \cite{rusapal}, and language expansion and combination of logics \cite{ruslu}.

A categorical approach to the combination of logics appears in \cite{ser1}, where the authors somehow refine Gabbay's fibration \cite{gab96} in a more general and purely syntactical context. Logical coproduct and amalgamation presented in \cite{ruslu} resemble, respectively, unconstrained and constrained fibration constructions of \cite{ser1}. In fact, at the language level they are exactly the same. However, the results of \cite{ruslu}, along with those presented here, show that the introduction of \emph{schema variables} is indeed unnecessary in order to combine different systems. Moreover, while fibration works as a categorical coproduct at the level of languages, this does not seem to be the case for the deductive systems as a whole.

Another categorical approach to various problems concerning the relationships among logics can be found in \cite{arnmarpin18,marpin17}. In those works, category theory is chosen as the primary tool even for the very basic definitions, and the main results are more concerned with algebraizability and extension of logics.

In this work we introduce a category in which essentially all the results obtained so far within this framework can be reviewed as categorical properties or constructions, thus trying to build a bridge between the algebraic/order-theoretic approach and the categorical one.

As a first step, we shall introduce a category, $\QM$, in which the objects are pairs of type $(Q,M)$, where $Q$ is a (unital) quantale and $M$ a left $Q$-module, and morphisms are pairs of maps in which the first coordinate is a quantale morphism and the second one is a quantale module morphism which exists thanks to the quantale morphism. We will discuss coproducts and pushouts in such a category. Then we shall introduce a non-full subcategory of  $\QM$ and show that it is equivalent to the category of propositional logics with interpretations as morphisms. Then we shall prove that the so-called ``logical coproduct'' introduced in \cite{ruslu} is actually the coproduct in this category. Last, we shall discuss the amalgamation of logics having a common fragment.

It is self-evident that the approach to deductive systems by means of quantales and $Q$-modules cannot be used as it is in the more complex and interesting realm of first order logics. However, propositional logics essentially represent the ``deductive skeleton'' of higher order ones and, therefore, knowing how to join together different reasoning systems in a single one opens the possibility of doing it at the first order level (possibly along with specific theories and their models), especially if we know that the new propositional logics are obtained by means of standard categorical procedures.

The paper is structured as follows.

In Section \ref{prel}, after recalling some preliminary notions and results on quantales and their modules, we will define the category $\QM$. In Section \ref{qmcoprodsec} we will study coproducts (Theorem \ref{qmcoprod}), and pushouts and amalgamated coproducts (Theorem \ref{amalgth}) in $\QM$.

The category $\DS_0$ of propositional deductive systems will be introduced in Section \ref{ds0sec}, along with its relationship with actual propositional logics, and a re-examination of the main results of \cite{rusapal} within such a new, refined, framework (Proposition \ref{ds0equiv}). Last, in Section \ref{ds0coprsec}, we will show, in Theorem \ref{ds0copr}, that the coproduct in $\DS_0$ does not coincide with the one of the same objects in $\QM$ and, in fact, is precisely the system constructed in \cite[Section 5]{ruslu}, along with the corresponding maps as morphisms. In the same section, we shall also discuss pushouts and amalgamation in $\DS_0$. It is well-known that monoids do not enjoy the amalgamation property, and this implies the same failure for quantales. Therefore, both the categories $\QM$ and $\DS_0$ do not have every amalgamated coproduct, and the same holds for pushouts. However, we will see a sufficient condition for a V-formation of deductive systems to be amalgamable (Theorem \ref{ds0amalg}), besides comparing the abstract categorical situation with the construction presented in \cite[Section 6]{ruslu}.

 \section{Preliminaries}
\label{prel}

First of all, we recall the definitions of sup-lattice, quantale, quantale module, and quantale module nucleus.

\begin{definition}\label{basic}
A \emph{sup-lattice} is a complete lattice in the category whose morphisms are join-preserving (or, equivalently, residuated) maps.

A \emph{(unital) quantale} is a monoid in the category of sup-lattices. In other words, a quantale is an algebraic structure $\la Q, \bigvee, \cdot, 1 \ra$ such that
\begin{enumerate}[(Q1)]
\item $\la Q, \bigvee \ra$ is a sup-lattice,
\item $\la Q, \cdot, 1 \ra$ is a monoid,
\item the multiplication distributes over arbitrary joins both from the left and from the right.
\end{enumerate}

A homomorphism between two quantales $Q$ and $R$ is a map $f: Q \to R$ which preserves arbitrary joins and the monoid structure. So, equivalently, a quantale homomorphism is a residuated monoid homomorphism. 

Let $Q$ be a quantale. A (left) \emph{$Q$-module} $M$, or a \emph{module over $Q$}, is a sup-lattice $\la M, \bigvee \ra$ endowed with an external binary operation, called \emph{scalar multiplication}, $\cdot: (a,x) \in Q \times M \mapsto a \cdot x \in M$ satisfying usual module axioms. Given two $Q$-modules $M$ and $N$, a $Q$-module morphism $f: M \to N$ is a multiplication-preserving sup-lattice morphism or, equivalently, a multiplication-preserving residuated map.
\end{definition}

\begin{definition}\label{nucleus}
A \emph{(left) $Q$-module nucleus} (or \emph{structural closure operator}) $\g$ over a $Q$-module $M$ is a closure operator -- namely, an extensive, monotone, idempotent map -- such that $a \cdot \g(x) \leq \g(a \cdot x)$, for all $a \in Q$ and $x \in M$. If $\g$ is a nucleus, we will denote by $M_\g$ the $\g$-closed system $\g[M]$, which is a left $Q$-module itself -- a homomorphic image of $M$, in fact -- with the operations defined by ${}^{M_\g}\bigvee X = \g ({}^M\bigvee X)$ and $a \cdot_{M_\g} x = \g(a \cdot_M x)$, for all $X \cup \{x\} \subseteq M_\g$ and $a \in Q$.\footnote{The subscripts and superscripts ``$M_\g$'' and ``$M$'' will never be used again, because the meaning of each operation will be clear from the context.}
\end{definition}
Right modules and bimodules are defined in a completely analogous way. We refer the reader to \cite{krupas,rosenthal,rusjlc,russajl,sol} for further information on quantales and their modules.

The next result is part of Lemma 6.1 of \cite{rusapal}, and is crucial for the definition of the category $\QM$.
\begin{lemma}\label{indmod}
Let $Q$ and $R$ be quantales and $h: Q \to R$ a quantale homomorphism. Then $h$ induces a structure of $Q$-module on every $R$-module $N$, with the following operation
\begin{equation}\label{starh}
\cdot_h: (a, x) \in Q \times N \lmapsto h(a) \cdot x \in N.
\end{equation}
\end{lemma}

The process described in the previous lemma is known as the \emph{restriction of scalars along $h$}. In fact, it defines a functor $(\ )_h: R\Mod \to Q\Mod$ between the categories of $R$-modules and $Q$-modules. Such a functor is both adjoint and coadjoint \cite[Theorem 6.7]{rusapal}. The left adjoint is of particular interest for us, as it involves the tensor product of modules, whose construction we recall hereafter.
\begin{theorem}[Theorem 6.3 of \cite{ruscorrapal}]\label{tensormqexists}
Let $M_1$ be a right $Q$-module and $M_2$ a left $Q$-module. Then the tensor product $M_1 \tensor_Q M_2$ of the $Q$-modules $M_1$ and $M_2$ exists. It is, up to isomorphisms, the quotient $\wp(M_1 \times M_2)/\theta_\rho$ of the free sup-lattice generated by $M_1 \times M_2$ by the (sup-lattice) congruence $\theta_\rho$ generated by the set
\begin{equation}\label{R}
\rho = \left\{
	\begin{array}{l}
		\left(\left\{\left(\bigvee X, y\right)\right\}, \bigcup_{x \in X}\{(x,y)\}\right) \\
		\left(\left\{\left(x, \bigvee Y\right)\right\}, \bigcup_{y \in Y}\{(x,y)\}\right) \\
		\left(\{(x \cdot_1 a, y)\}, \{(x,a \cdot_2 y)\}\right) \\
	\end{array} \right\vert
	\left.
	\begin{array}{l}
	X \subseteq M_1, y \in M_2 \\
	Y \subseteq M_2, x \in M_1 \\
	a \in Q \\
	\end{array}
	\right\}.
\end{equation}
\end{theorem}

Last, we recall that, given a quantale $Q$, a $Q$-module $M$, and $\vartheta \subseteq M^2$, an element $s$ of $M$ is called \emph{$\vartheta$-saturated} if, for all $(v,w) \in \vartheta$ and $a \in Q$, the following condition holds:
\begin{equation}\label{sateq}
av \leq s \iff aw \leq s;
\end{equation}
the quotient of $M$ over the congruence generated by $\vartheta$ is isomorphic to the set of the $\vartheta$-saturated elements of $M$ with suitable operations (see \cite{russajl}).


We conclude this section by defining the category of quantales and their modules.
\begin{definition}\label{qmcat}
The \emph{category of quantales and their modules}, denoted by $\QM$, is defined as follows:
\begin{itemize}
\item $\obj\QM$ (most often denoted simply by $\QM$, with an abuse of notation) is the class of pairs $(Q,M)$, where $Q$ is a unital quantale and $M$ a left $Q$-module;
\item for any two objects $(Q,M),(R,N) \in \QM$, a morphism between them is a pair $(h,f)$, where $h: Q \to R$ is a unital quantale homomorphism and $f: M \to N_h$ a $Q$-module homomorphism, $N_h$ being the $Q$-module structure on $N$ induced by $h$, as in Lemma \ref{indmod};
\item the identity morphisms are defined obviously as pairs of identities, and the composition of two morphisms $(h,f): (Q,M) \to (R,N)$ and $(k,g): (R,N) \to (S,P)$ is the morphism $(k \circ h, g \circ f): (Q,M) \to (S,P)$, each component of the pair being just the composition of functions.
\end{itemize}
\end{definition}


\section{Coproducts and pushouts in $\QM$}
\label{qmcoprodsec}

The category $\QM$ can be seen also as a non-full subcategory of the product category $\cat Q \times \SL$, of the category of quantales with the one of sup-lattices. However, due to the strong condition imposed on morphisms, it is not surprising that $\QM$ has a quite different behaviour with respect to $\cat Q \times \SL$. In this paper, we introduced this category as a (not too) broader abstract framework for the category of deductive systems. For this reason, in this paper we will concentrate our attention to amalgamation and coproduct, which are particularly relevant to logic. The coproduct of unital quantales has been described in \cite{liang}. It is the quotient of the powerset of the coproduct of the monoid reducts of the given quantales, over the smallest congruence that identifies subsets of the factors with their join in the respective quantales. The product of quantales is given simply by the Cartesian product with coordinatewise operations. Regarding sup-lattices and quantale modules, they have analogous products, while coproducts have exactly the same objects as products \cite[Lemma 5.1]{galtsi}. 

The coproduct in the category $\QM$ is strongly connected to the tensor product of quantale modules discussed in \cite{rusapal} and \cite{ruscorrapal}, as we can see from the following characterization.


\begin{theorem}\label{qmcoprod} 
Let $F = \{(Q_i,M_i)\}_{i \in I}$ be a family of objects of $\QM$ and let $R = \coprod_{i \in I} Q_i$ be the coproduct of the $Q_i$'s in the category of unital quantales (see also \cite{liang}). Then the coproduct of $F$ in $\QM$ is the object
$$\left(R, \coprod_{i \in I} R \tensor_{Q_i} M_i\right),$$
with associated family of morphisms $\{(\nu_i,\mu_i (1 \tensor \iota_{M_i}))\}$ where, for each $i \in I$, $\nu_i$ is the embedding of $Q_i$ in $R$, $1 \tensor \iota_{M_i}$ is the $Q_i$-module morphism defined by $x \in M_i \mapsto 1 \tensor x \in R \tensor_{Q_i} M_i$ as in \cite[Section 6]{rusapal}, and $\mu_i$ is the $R$-module embedding of $R \tensor_{Q_i} M_i$ in the $R$-module coproduct of the family $\{R \tensor_{Q_i} M_i\}_{i \in I}$ as in \cite[Proposition 4.13]{rusjlc}.
\end{theorem}
\begin{proof}
Let $(S,N) \in \QM$ and $\{(h_i,f_i): (Q_i, M_i) \to (S,N)\}$ be a family of morphisms. Since $R$ is the quantale coproduct of the $Q_i$'s, there exists a quantale morphism $h: R \to S$ such that $h\nu_i=h_i$ for all $i \in I$.

By \cite[Lemma 6.6]{rusapal}, the sup-lattices $\hom_R(R,N_h)$ and $N$ are isomorphic; moreover, such an isomorphism becomes a $Q_i$-module isomorphism between the structure of $Q_i$-module on $\hom_R(R,N_h)$ defined in \cite[Lemma 6.4]{rusapal} (here $R$ is to be thought of as an $R$-$Q_i$-bimodule) and $(N_h)_{\nu_i} = N_{h_i}$. Consequently, we have $\hom_{Q_i}(M_i,N_{h_i}) \cong_{\SL} \hom_{Q_i}(M_i, \hom_R(R,N_h))$. The latter, in turn, is isomorphic to $\hom_R(R \tensor_{Q_i} M_i, N_h)$ by \cite[Lemma 6.5]{rusapal}. Therefore we have $\hom_{Q_i}(M_i,N_{h_i}) \cong_{\SL} \hom_R(R \tensor_{Q_i} M_i, N_h)$.

Then we have an $R$-module morphism $f_i': R \tensor_{Q_i} M_i \to N_h$ canonically associated to each $Q_i$-module morphism $f_i: M_i \to N_{h_i}$. The $R$-module $M = \coprod_{i \in I} R \tensor_{Q_i} M_i$ is the coproduct of the family $\{R \tensor_{Q_i} M_i\}_{i \in I}$ in $R\Mod$, hence there exists an $R$-module morphism $f: M \to N_h$ such that $f'\mu_i = f_i'$ for all $i$.

Taking a closer look at all those isomorphisms (the proof of \cite[Lemma 6.5]{rusapal} was actually presented in \cite[Theorem 4.7.6]{rusthesis}), it is not hard to see that, for all $i \in I$, $r \in R$, and $x \in M_i$, $f_i'(r \tensor x) = r \cdot f_i(x)$, whence
$$(f\mu_i(1 \tensor \iota_{M_i}))(x) = f_i'(1 \tensor x) = f_i(x),$$
which completes the proof.
$$\begin{tikzcd}
R \arrow[rrd, "h"]  &   & &  & M  \arrow[rrd, "f"] & &   \\
                                          & & S &  & R \tensor_{Q_i} M_i \arrow[u, "\mu_i"] \arrow[rr, "f_i'"]  & & N \\
Q_i \arrow[uu, "\nu_i"] \arrow[rru, "h_i"'] &  & &  & M_i \arrow[u, "1 \tensor \iota_{M_i}"] \arrow[rru, "f_i"']     &  &
\end{tikzcd}$$
\end{proof}

\begin{proposition}\label{tensoremb}
Let $h: Q \to R$ be a quantale morphism and $M \in Q\Mod$. If there exists an embedding $f: M \to N_h$ for some $N \in R\Mod$, then the $Q$-module homomorphism
$$\mu: x \in M \mapsto 1 \tensor x \in R \tensor_Q M$$
is an embedding.
\end{proposition}
\begin{proof}
Let us use the following notations: $M' = f[M]$ and $\top' = \bigvee M'$.

Let, for all $x \in M$, $\ov x$ be the following element of $\wp(R \times M)$:
\begin{equation}\label{ovx}
\ov x = \{(r, y) \in R \times M \mid 
r \cdot f(y) \leq f(x)\}.
\end{equation}

Recalling that sup-lattices can be seen also as modules over the two-element quantale $\{\bot,1\}$, we shall prove that, for each $x \in M$, $\ov x$ is a saturated element of the relation defined as in (\ref{R}) by proving that condition (\ref{sateq}) is verified for all the pair types in (\ref{R}). The scalar $a$ in (\ref{sateq}) shall be dropped because it is actually $1$, while the case of $a = \bot$ is trivial. 

Let us check (\ref{sateq}) for the first type of pairs in (\ref{R}). Let $A \subseteq R$ and $y \in M$; then
$$\begin{array}{l}
\{(\bigvee A, y)\} \subseteq \ov x \\
\iff 
\bigvee A \cdot f(y) \leq f(x) \\
\iff 
r \cdot f(y) \leq f(x) \\
\iff \bigcup_{r \in A}\{(r, y)\} \subseteq \ov x.
\end{array}.$$

As for the second type, let $r \in R$ and $Y \subseteq M$. Then we have:
$$\begin{array}{l}
\{(r,\bigvee Y)\} \subseteq \ov x \\
\iff 
r \cdot f(\bigvee Y) = \bigvee\limits_{y \in Y} r \cdot f(y) \leq f(x) 
\\
\iff 
\forall y \in Y \ (r \cdot f(y) \leq f(x))\\
\iff \forall y \in Y \ (\{(r, y)\}  \subseteq \ov x)\\
\iff \bigcup_{y \in Y}\{(r, y)\}  \subseteq \ov x.
\end{array}$$

Last, let us consider $r \in R$, $q \in Q$, and $y \in M$. 
The equality $r \cdot f(q \cdot y) = (r \cdot h(q))\cdot f(y)$ implies that $r \cdot f(q \cdot y) \leq f(x)$ if and only if $(r \cdot h(q))\cdot f(y) \leq f(x)$, and therefore $\{r, q \cdot y)\} \subseteq \ov x$ if and only if $\{(r \cdot h(q), y)\} \subseteq \ov x$. This finally proves that $\ov x$ is saturated in $R \times M$ with respect to the relation which determines the tensor product.

As a final step, let us show that the mapping $x \mapsto \ov x$ is injective. For all $x \in M$, $(1,x) \in \ov x$ and, for any $y \in M$, if $\ov x = \ov y$ then $(1,x) \in \ov y$ and $(1,y) \in \ov x$, which implies that $1 \cdot f(x) = f(x) \leq f(y)$ and $1 \cdot f(y) = f(y) \leq f(x)$, i.e., $f(x) = f(y)$. Then, by the injectivity of $f$, $x = y$ and the thesis follows.

\end{proof}

\begin{theorem}\label{amalgth}
Let $\mathbf A = (Q_1, M_1) \stackrel{(h_1,f_1)}{\longleftarrow} (Q,M) \stackrel{(h_2,f_2)}{\longrightarrow} (Q_2,M_2)$ be a span in $\QM$ and $R$ a quantale such that
$$\mathbf P: \quad \begin{tikzcd}
                                       & Q_1 \arrow[rd, "k_1"]  &   \\
Q \arrow[ru, "h_1"] \arrow[rd, "h_2"'] &                        & R \\
                                       & Q_2 \arrow[ru, "k_2"'] &  
\end{tikzcd}$$
is a pushout square in the category of quantales. If, for $i=1,2$, $M_i$ embeds in $(P_i)_{k_i}$ as a $Q_i$-module, for some $R$-module $P_i$, then the pushout of $\mathbf A$ exists in $\QM$ and it is, up to isomorphism, $(R,R \tensor_Q N)$, where $N$ is the pushout of the $Q$-module span $(M_1)_{h_1} \stackrel{f_1}{\longleftarrow} M \stackrel{f_2}{\longrightarrow} (M_2)_{h_2}$.
 
Moreover, if $\mathbf P$ is the diagram of an amalgamated coproduct (i.e. the morphisms in $\mathbf P$ are embeddings), and $f_1$ and $f_2$ are embeddings, then $(R,R \tensor_Q N)$ is the amalgamated coproduct of $\mathbf A$ in $\QM$.
\end{theorem}
\begin{proof}
First of all, let us observe that, for $i = 1, 2$, the following sup-lattice isomorphisms hold:
$$\begin{array}{ll}
\hom_R(R \tensor_Q M, R \tensor_{Q_i} M_i) &  \\
\cong \hom_Q(M, \hom_R(R, R \tensor_{Q_i} M_i)) & \text{(by \cite[Theorem 4.7.6]{rusthesis})} \\
\cong \hom_Q(M, R \tensor_{Q_i} M_i)  &  \text{(by \cite[Theorem 4.7.8]{rusthesis})} \\
\cong \hom_Q(M, \hom_{Q_i}(Q_i, R \tensor_{Q_i} M_i)) & \text{(by \cite[Theorem 4.7.8]{rusthesis})} \\
\cong \hom_{Q_i}(Q_i \tensor_Q M, R \tensor_{Q_i} M_i) & \text{(by \cite[Theorem 4.7.6]{rusthesis})} \\
\cong \hom_{Q_i}(Q_i \tensor_Q M, \hom_R(R, R \tensor_{Q_i} M_i)) & \text{(by \cite[Theorem 4.7.8]{rusthesis})} \\
\cong \hom_R(R \tensor_{Q_i}(Q_i \tensor_Q M), R \tensor_{Q_i} M_i) & \text{(by \cite[Theorem 4.7.6]{rusthesis})}
\end{array}$$

Let $N$ be the pushout of the $Q$-module span
$$(M_1)_{h_1} \stackrel{f_1}{\longleftarrow} M \stackrel{f_2}{\longrightarrow} (M_2)_{h_2},$$
with associated morphisms $n_i: (M_i)_{h_i} \to N$, $i=1,2$, and let $f_i': R \tensor_Q M \to R \tensor_{Q_i} M_i$, $i = 1,2$, be the morphisms obtained via the above correspondence from the image of the morphism $f_i$ under the functor $R \tensor_{Q_i} (Q_i \tensor_Q {}_{\text{--}})$ or, which is the same, the image of $f_i$ under the left adjoint to the functor $( \ )_{k_ih_i} = ( \ )_{h_i} \circ ( \ )_{k_i}$. 

Since $R$ is the pushout of $h_1$ and $h_2$, whenever we have a pair of quantale morphisms $l_i: Q_i \to S$ such that $l_1h_1 = l_2h_2$, there exists a quantale morphism $k: R \to S$ such that $kk_i=l_i$, $i = 1,2$, as in the following diagram:
\begin{equation}\label{push1}
\begin{tikzcd}
                                       & Q_1 \arrow[rd, "k_1"] \arrow[rrd, "l_1", bend left]    &                  &   \\
Q \arrow[rd, "h_2"'] \arrow[ru, "h_1"] &                                                        & R \arrow[r, "k"] & S \\
                                       & Q_2 \arrow[ru, "k_2"'] \arrow[rru, "l_2"', bend right] &                  &  
\end{tikzcd}.
\end{equation}
Now, referring ourselves to the next diagram, if we have an $S$-module $L$ and morphisms $g_i: M_i \to L_{l_i}$, $i = 1,2$, such that $g_1f_1 = g_2f_2$, then the tensor functors will give us two $R$-module morphisms $g_i': R \tensor_{Q_i} M_i \to L_k$, $i=1,2$, such that $g_1'f_1' = g_2'f_2'$. Since $N$ is a $Q$-module pushout, there exists a $Q$-module morphism $g: N \to L_{l_ih_i}$ such that $gn_i = g_i$, for $i = 1,2$. Applying the tensor functor again, we get then an $R$-module morphism $g': R \tensor_Q N \to L_k$ such that $g'n_i' = g_i'$, for $i = 1,2$. Now, since the $M_i$'s embed in some $R$-modules by hypothesis, by Proposition \ref{tensoremb}, the diagram (\ref{push2}) commutes and, therefore, $(R,R \tensor_Q N)$ is the desired pushout. This proves the first assertion.
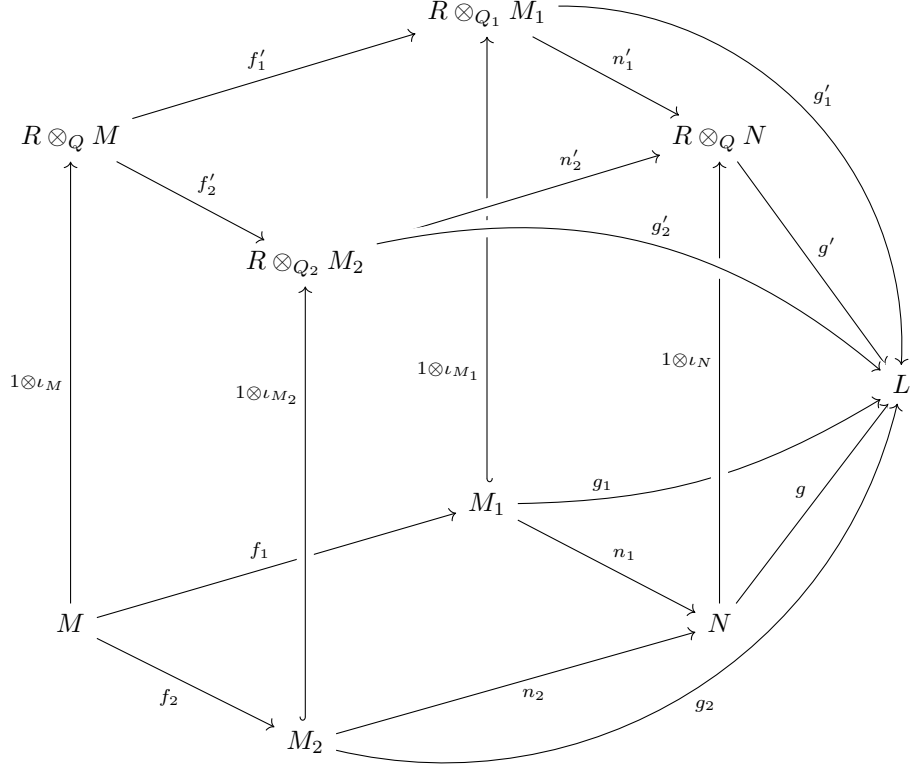
\begin{equation}\label{push2}
\begin{tikzcd}[row sep=3em,column sep=1.5em]
{}& {}&{} & R \tensor_{Q_1} M_1 \arrow[rrrrddd, "g_1'", bend left=47] \arrow[rrd, "n_1'"] \arrow[dddd, hookleftarrow, "1 \tensor \iota_{M_1}"', near end]    &  {}&{} & {}& {} \\
R \tensor_Q M \arrow[rrd, "f_2'", crossing over] \arrow[rrru, "f_1'"]  &  {} & {}& {}  &{} & R \tensor_Q N  \arrow[rrdd, "g'"]  & {}&{} \\
{}&{} & R \tensor_{Q_2} M_2 \arrow[rrru, "n_2'", near end, crossing over]  \arrow[rrrrrd, "g_2'", bend left=26,crossing over]   & {} &{} &{}\arrow[u] &{} &{} \\
{}& {}& {}& {}& {}& {}& {}& L\\
{} &	{}&		{}& M_1    \arrow[rrd, "n_1"] \arrow[rrrru, "g_1", bend right=15, near start]&     {}&   {} &{} &{} \\
M \arrow[uuuu, "1 \tensor \iota_M"] \arrow[rrd, "f_2"'] \arrow[rrru, "f_1"] &    {}& {} & {}  & {}& N \arrow[rruu, "g"] \arrow[uuu, "1 \tensor \iota_N", crossing over, no head, near end] & {}& {}\\
 {}& {}& M_2 \arrow[rrrrruuu, "g_2"', crossing over, bend right=45] \arrow[rrru, "n_2"'] \arrow[uuuu, "1 \tensor \iota_{M_2}", near end, crossing over, hook] & {}& {}& {}& {}&{}
\end{tikzcd}
\end{equation}

With the given additional hypotheses, $f_1$, $f_2$, $n_1$, $n_2$, $1 \tensor \iota_M$, and $1 \tensor \iota_N$ are embeddings too, therefore $(R,R \tensor_Q N)$ is the amalgamation of the span $\mathbf A$. 
\end{proof}

\section{The subcategory of deductive systems}
\label{ds0sec}

In this section, we shall define the category $\DS_0$ as a subcategory of $\QM$, and review some results from \cite{rusapal} and \cite{ruslu} in the light of the ones above. First, we shall recall some very basic definitions.

As usual, by a \emph{propositional language} we mean a pair $\lang = \la L, \nu \ra$ consisting of a set $L$ and a map $\nu: L \to \omega$. The elements of $L$ are called \emph{connectives}, and the image of a connective under $\nu$ is called its \emph{arity}; nullary connectives are most often called \emph{constant symbols} or simply  \emph{constants}.

Given a propositional language $\lang$ and a denumerable set $\V = \{x_n \mid n < \omega\}$ of \emph{propositional variables}, the set $\fml$ of $\lang$-formulas is defined recursively in the usual manner:
\begin{enumerate}[(F1)]
\item every propositional variable and every constant symbol is an $\lang$-formula;
\item if $f$ is a connective of arity $\nu(f) > 0$ and $\phi_1, \ldots, \phi_{\nu(f)}$ are $\lang$-formulas, then $f\phi_1\ldots\phi_{\nu(f)}$ -- usually denoted by $f(\phi_1, \ldots, \phi_{\nu(f)})$ -- is an $\lang$-formula;
\item all $\lang$-formulas are built by finitely many iterative applications of (F1) and (F2).
\end{enumerate}
From a universal algebraic viewpoint, $\fml$ can be also defined as the absolutely free algebra over the signature $\lang$. The \emph{substitution monoid $\sfm$ over $\lang$} is the monoid of $\lang$-endomorphisms of $\fml$. We remark that, since $\fml$ is a term algebra, each substitution is completely determined by its values on the variables. 

Starting from formulas, it is possible to define
\begin{itemize}
\item the set $\Eq$ of $\lang$-equations as $\fml^2$, and
\item for all $T \subseteq \omega^2$, the set $\seq_T$ of \emph{sequents} closed under the \emph{types} in $T$ as $\bigcup\limits_{(m,n) \in T} \fml^m \times \fml^n$.
\end{itemize}
An \emph{inference rule} over $\lang$ is a pair $(\Phi,\psi)$ where $\Phi$ is a set of formulas, equations, or sequents of a fixed type, and $\psi$ is a single formula, equation, or sequent of the same type. Then we say that $\phi$ is \emph{directly derivable} from $\Psi$ by the rule $(\Phi,\psi)$ if there is a substitution $\sigma$ such that $\sigma \psi = \phi$ and $\sigma[\Phi] \subseteq \Psi$. An inference rule $(\Phi,\psi)$ is usually denoted by $\frac{\Phi}{\psi}$.

An \emph{axiom} in the language $\lang$ is simply a formula (or an equation, or a sequent) in $\lang$.

\begin{definition}\label{consrel}
A \emph{propositional logic} for short, $S$ over a given language $\lang$, is defined by means of a (possible infinite) set of inference rules and axioms. It consists of the pair $S = (D, \vdash)$, where $\vdash$ is a subset of $\wp D \times D$ -- $D$ being the set of $\lang$-formulas, the set of $\lang$-equations, or a set of $\lang$-sequents closed under type -- defined by the following condition: $\Phi \vdash \psi$ iff $\psi$ is contained in the smallest set of formulas that includes $\Phi$ together with all substitution instances of the axioms of $S$, and is closed under direct derivability by the inference rules of $S$. The relation $\vdash$ is called the \emph{consequence relation} of $S$.
\end{definition}

Now we recall the representation of propositional logics as quantale modules. If $D$ is $\fml$, $\Eq$, or any $\seq_T$, we have a left monoid action from $\sfm$ to $D$. By applying the left adjoints to the forgetful functors from quantales to monoids and from sup-lattices to sets respectively, we obtain a quantale $\Sfm$ and a left $\Sfm$-module $\wp D$. Thanks to this change of perspective, one can see a consequence relation $\vdash$ as a $\Sfm$-module nucleus on $\wp D$, as the following result shows.
\begin{proposition}[Lemma 3.5 of \cite{galtsi}]
For any consequence relation $\vdash$ on $\wp D$, the mapping
$$\g_\vdash: \Phi \in \wp D \mapsto \bigcup_{\Phi \vdash \Psi} \Psi \in \wp D$$
is a $\Sfm$-module nucleus. Reciprocally, for any $\Sfm$-module nucleus $\g$ on $\wp D$, the relation
$$\Phi \vdash_\g \Psi \iff \Psi \subseteq \g(\Phi)$$
is a consequence relation on $\wp D$.  

Moreover, $\vdash_{\g_\vdash} = \ \vdash$ and $\g_{\vdash_\g} = \g$, for any consequence relation $\vdash$ and for any nucleus $\g$.
\end{proposition}
Given a consequence relation $\vdash$ with associated nucleus $\g$, the $\g$-closed system $\wp D_\g$ coincides with the lattice of theories $\th_\vdash$ of $\vdash$. Thanks to the previous result, we can think of a consequence relation either as a binary relation or as a $\Sfm$-module nucleus, and we shall use either one of the notations depending on convenience. Similarly, we shall indifferently denote the lattice of theories by $\wp D_\g$, $\th_\vdash$ or $\th_\g$, and we will rather call it the \emph{module of theories}. 

Such a representation was introduced by Galatos and Tsinakis in \cite{galtsi} with the aim of algebraically characterizing interpretations between logical systems over the same language, such as Glivenko interpretation of classical into intuitionistic logic \cite{gli}, and some cases of algebraizability \cite{blpi}. Eventually, the present author extended that result to the case of systems with different underlying languages. Such an extension obviously required the addition of a translation of the connectives of one language into the other. So we shall now recall the definition of language translation and two results about it from \cite{rusapal}. They are crucial for the definition of the category $\DS_0$.

\begin{definition}\label{translation}
Let $\lang = \la L, \nu \ra$ and $\lang' = \la L', \nu' \ra$ be two propositional languages. Assume that for each connective $f \in L$ there exists a derived operation $f'$ on $\fmlp$ of arity $\nu(f)$. If we denote by $\lang^{\fmlp}$ the language composed of such operations, the structure $\mathit{Fm}'_{\lang} = \la \fmlp, \lang^{\fmlp} \ra$ is an $\lang$-algebra. In this case, a map $\tau: \fml \lto \fmlp$ is called a \emph{language translation} of $\lang$ into $\lang'$ if
\begin{enumerate}[(i)]
\item $\tau^{-1}(x) = \{x\}$ for any variable $x$,
\item $\tau$ is an $\lang$-homomorphism, that is,
$$\tau(f(\phi_1, \ldots, \phi_{\nu(f)})) = f'(\tau(\phi_1), \ldots, \tau(\phi_{\nu(f)})),$$
for all $f \in L$ and $\phi_1, \ldots, \phi_{\nu(f)} \in \fml$.
\end{enumerate}
\end{definition}

\begin{lemma}[Lemma 3.3 of {{\cite{rusapal}}}]\label{qinq'}
Let $\tau$ be a language translation of $\lang = \la L, \nu \ra$  into $\lang' = \la L', \nu' \ra$. The following hold:
\begin{enumerate}[(i)]
\item $\tau$ induces a monoid homomorphism $\ov\tau: \sfm \to \sfmm$. More concretely, for each $\sigma \in \sfm$, let $\sigma'$ be the substitution uniquely determined by the map $\tau \circ \sigma_{\restr \V} \in \fmlp^{\V}$. Then $\ov\tau$ is defined by $\ov \tau(\sigma)=\sigma'.$
\item $\ov \tau$ is injective (resp.: surjective) if and only if $\tau$ is.
\item $\tau$ commutes with the substitutions in $\sfm$ in the following sense: $\tau(\sigma(\phi)) = \ov\tau(\sigma)(\tau(\phi))$ for all $\sigma \in \sfm$ and $\phi \in \fml$.
\end{enumerate}
\end{lemma}

Now recall that, for any language $\lang$, the set $V=\{\sigma \in \sfm \mid \sigma[\V] \subseteq \V\}$ is the universe of a submonoid $\VV$ of $\sfm$. Furthermore, if $\VV$ and $\VV'$ are two such submonoids corresponding to the languages $\lang$ and $\lang'$, then $\VV\cong\VV'$. In the sequel, we identify all these monoids and denote all of them by $\VV$.

The first part of the following result is an immediate consequence of the previous lemma, while the second part is Corollary 6.9 of \cite{rusapal}.

\begin{proposition}\label{transchar2}
Let $\tau$ be a language translation of $\lang = \la L, \nu \ra$  into $\lang' = \la L', \nu' \ra$. Then $\tau$ induces a quantale homomorphism $\ov\tau: \Sfm \to \Sfmm$.

Reciprocally, a quantale homomorphism $h: \Sfm \to \Sfmm$ is induced by a language translation of $\lang$ in $\lang'$ if and only if it satisfies the following conditions:
\begin{enumerate}[(i)]
\item $h$ preserves the property of being completely join prime;
\item if $\Sigma \in \Sfmm$ is completely join prime and multiplicatively idempotent, then $h^{-1}(\Sigma)$ is either empty or is comprised of completely join prime idempotent elements of $\Sfm$;
\item $h^{-1}(\Sigma) = \{\Sigma\}$ for all $\Sigma \in \wp\VV$.
\end{enumerate}
\end{proposition}

We refer to a quantale homomorphism $h: \Sfm \to \Sfmm$ induced by a language translation as a \emph{quantale translation}.

As an instance of \cite[Definition 3.7]{rusapal}, we get the following
\begin{definition}\label{ints}
Let $\lang$ and $\lang'$ be languages, $\tau: \lang \to \lang'$ a language translation with associated monoid homomorphism $h: \sfm \to \sfmm$, and $S = \la D, \vdash \ra$  and $T = \la D', \vdash'\ra$ be deductive systems over $\lang$ and $\lang'$ respectively.

A map $\iota: D \lto \wp D'$ is said to be \emph{$h$-action-invariant} if $\iota(\sigma(\phi)) = h(\sigma) \cdot\iota(\phi)$ for all $\sigma \in \sfm$ and $\phi \in D$.

An $h$-action-invariant map $\iota: D \to \wp D'$ is called an \emph{interpretation} \emph{via $h$} of $\vdash$ in $\vdash'$ if, for all $\Phi \cup \{\psi\} \subseteq D$,
\begin{equation*}\label{neqint}
\Phi \vdash \psi \quad \textrm{implies} \quad \iota[\Phi] \vdash' \iota(\psi).
\end{equation*}
An interpretation is called a \emph{conservative interpretation}, or a \emph{representation}, if
\begin{equation*}\label{neqintcon}
\Phi \vdash \psi \quad \textrm{if and only if} \quad \iota[\Phi] \vdash' \iota(\psi).
\end{equation*}
 Given another translation $\tau': \lang' \to \lang$, with associated monoid homomorphism $k: \sfmm \to \sfm$, two representations $\iota: D \to \wp D'$, via $h$, and $\iota': D' \to \wp(D)$, via $k$, are said to form an \emph{equivalence} if, for all $\phi \in D'$,
\begin{equation*}\label{eqeq}
\phi \dashv'\vdash' \iota[\iota'(\phi)].
\end{equation*}
\end{definition}

We can now set the following
\begin{definition}
Let us denote by $\DS_0$ the \emph{category of propositional deductive systems}, namely, the subcategory of $\QM$ whose objects are pairs of type $(\Sfm, \th)$, where $\Sfm$ is the quantale of substitutions of some propositional language $\lang$ and $\th$ is the $\Sfm$-module of theories of a deductive system over $\lang$, and whose morphisms have quantale morphisms induced by language translations in the first component. More precisely, a morphism in $\DS_0$ is a pair $(h,f)$ in which $h: \Sfm \to \Sfmm$ is a quantale translation, i.e., verifies conditions (i)--(iii) of Proposition \ref{transchar2}.
\end{definition}
Although quite obvious, it is worth observing explicitly that $\DS_0$ is not a full subcategory of $\QM$.

By Theorem 7.1 and Corollary 7.2 of \cite{rusapal}, the category whose objects are propositional logics and whose morphisms are pairs translation-interpretation between them is equivalent to $\DS_0$. More precisely, we have the following immediate result.
\begin{proposition}\label{ds0equiv}
Let $\lang$ and $\lang'$ be languages, $S = \la D, \vdash \ra$ be an $\lang$-deductive system and $T = \la D', \vdash'\ra$ an $\lang'$-deductive system. Then:\begin{enumerate}[(i)]
\item $S = \la D, \vdash \ra$ is interpretable in $T = \la D', \vdash'\ra$ if and only if
$$\hom_{\DS_0}((\Sfm, \th_{\vdash}), (\Sfmm, \th_{\vdash'})) \neq \varnothing;$$
\item $S = \la D, \vdash \ra$ is conservatively interpretable in $T = \la D', \vdash'\ra$ if and only if there exists $(h,f) \in \hom_{\DS_0}((\Sfm, \th_{\vdash}), (\Sfmm,\th_{\vdash'}))$ such that $f$ is injective;
\item $S = \la D, \vdash \ra$ and $T = \la D', \vdash'\ra$ are equivalent if and only if there exist
$$\begin{array}{l}
(h,f) \in \hom_{\DS_0}((\Sfm, \th_{\vdash}), (\Sfmm,\th_{\vdash'})), \textrm{ and} \\
(k,g) \in \hom_{\DS_0}((\Sfmm,\th_{\vdash'}),(\Sfm, \th_{\vdash}))
\end{array}$$ such that $g$ and $f$  are isomorphisms. 
\end{enumerate}
\begin{proof}
By Proposition \ref{transchar2}, language translations are in bijective correspondence with quantale morphisms which are first coordinates of $\DS_0$-morphisms. So, items (i) and (ii) follow readily from \cite[Theorem 7.1]{rusapal}, while item (iii) is an equally immediate consequence of \cite[Corollary 7.2]{rusapal}. 
\end{proof}
\end{proposition}

\section{Coproduct, pushout and amalgamation in $\DS_0$}
\label{ds0coprsec}

In \cite[Section 5]{ruslu}, the author constructed the so-called ``logical coproduct" of deductive systems and proved that the given systems are conservatively interpretable in it \cite[Theorem 5.3]{ruslu}. As clearly stated in that paper, the name logical coproduct was suggested by the relationship between the new logic and its components, but it was not clear whether that construction really corresponded to a coproduct in some category. The main result of this section shows that the answer is positive: $\DS_0$ is the right categorical setting and the logical coproduct of propositional logics is indeed a coproduct in the narrow categorical sense.

For the sake of readability, the original construction had been given for a pair of systems and languages, but it was also observed that everything holds for arbitrary families of systems and languages. So, let us briefly recall it in such a more general setting.

Given a family of propositional languages $\{\lang_i\}_{i \in I}$, the quantale of substitutions $\wp\sfm$ of their disjoint union $\lang$ is the coproduct of the family of quantales $\{{\Sfm}_i\}_{i \in I}$. Now, for each $i \in I$, let $(D_i, \vdash_i)$ be a deductive system, of the same fixed type, over $\lang_i$, with associated nucleus $\g_i$ and module of theories $\th_i = (\wp D_i)_{\g_i}$. The logical coproduct of the given systems $(E, \vdash)$ is the system defined over the $\lang$-domain $E$ of the same type of the $D_i$'s, and whose consequence relation $\vdash$ is defined by means of the union of the axioms and rules of the family $\{\vdash_i\}_{i \in I}$.

With the above notations, let also $\th$ be the $\Sfm$-module of theories of $(E,\vdash)$. Then, combining Theorem \ref{qmcoprod} with the results of \cite[Section 5]{ruslu}, we can see that the logical coproduct is indeed the coproduct in $\DS_0$ while it is not necessarily isomorphic to the one in $\QM$, as the following result and counterexample show.
\begin{theorem}\label{ds0copr}
$(\Sfm,\th)$ is, up to an isomorphism, the coproduct of the family $\{({\Sfm}_i, \th_i)\}_{i \in I}$ in $\DS_0$.
\end{theorem}
\begin{proof}
Let $\lang'$ be a propositional language and $(D',\vdash')$ be an $\lang'$-system (of any type) with $\g'$ and $\th'$ being, respectively, the associated module nucleus on $\wp D'$ and the module of theories. Suppose that, for each $i \in I$, there exists a language translation $t_i: \lang_i \to \lang'$ and an interpretation $f_i: \th_i \to \th'$ with respect to $t_i$ (we shall denote by $t_i$ also the quantale morphisms induced by the given translations). Since $\Sfm$ is the quantale coproduct of the ${\Sfm}_i$'s and all the associated embeddings are induced by language translations, it is easy to see that the morphism $t: \Sfm \to \Sfmm$ extending the $t_i$'s to $\Sfm$ is also induced by a language translation. Moreover, by \cite[Theorem 7.1]{rusapal}, given such a translation, any $\Sfm$-module morphism from $\th$ to $\th_t'$ is indeed induced by an interpretation, so we just need to show that there exists a morphism $f: \th \to \th_t'$ such that $fe_i = f_i$ for all $i \in I$, where $e_i: \th_i \to \th$ is the embedding of \cite[Theorem 5.3]{ruslu}.

Let us denote by $(\Sfmm, M)$ the coproduct of $\{({\Sfm}_i,\th_i)\}_{i \in I}$ in $\QM$. By Theorem \ref{qmcoprod}, $M = \coprod_{i \in I} \Sfm \tensor_{Q_i} \th_i$ and, for each $i \in I$, there exists a ${\Sfm}_i$-module embedding of $\th_i$ into $M_{t_i}$. Now let us consider the following commutative diagram
\begin{equation}\label{coprdiag}
\begin{tikzcd}
\th_i \arrow[rdd, "g_i"', bend right] \arrow[rd, "e_i"] \arrow[rr, "f_i"] &                                                  & \th' \\
                                                                          & \th \arrow[ru, "f", dashed]                      &      \\
                                                                          & M \arrow[u, "g"'] \arrow[ruu, "g'"', bend right] &     
\end{tikzcd}
\end{equation}
where all the $i$-indexed morphisms are ${\Sfm}_i$-module morphisms, while $g$ and $g'$ are the $\Sfm$-module morphisms extending, respectively, the $e_i$'s and the $f_i$'s to $M$. We want to prove that the $\Sfm$-module morphism $f$ exists. In order to do that, let us prove that $\ker g \subseteq \ker g'$, so let us consider a pair $((m_i)_{i \in I}, (n_i)_{i \in I}) \in \ker g$. Each $m_i$ and each $n_i$ are joins of tensors, say
$$m_i = \bigvee_{j_i \in J_i} \Sigma_{j_i} \tensor_{{\Sfm}_i} \Phi_{j_i} \quad \text{ and } \quad n_i = \bigvee_{k_i \in K_i} \Xi_{k_i} \tensor_{{\Sfm}_i} \Psi_{k_i},$$
with the $\Phi_{j_i}$'s and $\Psi_{k_i}$'s in $\th_i$ for all $i \in I$, and $\Sigma_{j_i}, \Xi_{k_i} \in \Sfmm$ for all indexes. We have the following:
$$\begin{array}{l} 
g\left(\left(m_i\right)_{i\in I}\right) = g\left(\left(n_i\right)_{i \in I}\right) \iff \\
g\left(\left(\bigvee_{j_i \in J_i} \Sigma_{j_i} \tensor_{{\Sfm}_i} \Phi_{j_i}\right)_{i\in I}\right) = g\left(\left(\bigvee_{k_i \in K_i} \Xi_{k_i} \tensor_{{\Sfm}_i} \Psi_{k_i}\right)_{i \in I}\right) \iff \\
\bigvee_{i \in I} \left(\bigvee_{j_i \in J_i} \Sigma_{j_i} \cdot e_i\left(\Phi_{j_i}\right)\right) = \bigvee_{i \in I} \left(\bigvee_{k_i \in K_i} \Xi_{k_i} \cdot e_i\left(\Psi_{k_i}\right)\right) \iff \\
\bigvee_{i \in I} \left(\bigvee_{j_i \in J_i} \Sigma_{j_i} \cdot \Phi_{j_i}\right) = \bigvee_{i \in I} \left(\bigvee_{k_i \in K_i} \Xi_{k_i} \cdot \Psi_{k_i}\right) \iff \\
\bigvee_{i \in I} \left(\bigvee_{j_i \in J_i} \Sigma_{j_i} \cdot \Phi_{j_i}\right) \dashv\vdash \bigvee_{i \in I} \left(\bigvee_{k_i \in K_i} \Xi_{k_i} \cdot \Psi_{k_i}\right).
\end{array}$$
Now, let $\Gamma = \bigvee_{i \in I} \left(\bigvee_{j_i \in J_i} \Sigma_{j_i} \cdot \Phi_{j_i}\right)$, $\Delta = \bigvee_{i \in I} \left(\bigvee_{k_i \in K_i} \Xi_{k_i} \cdot \Psi_{k_i}\right)$, and recall that, by \cite[Section 5]{ruslu}, $\vdash$ is the smallest $\sfm$-invariant consequence relation on $E$ verifying all the axioms and rules of all the $\vdash_i$'s. Therefore, there exists a $\vdash$-proof, i.e., a finite sequence of entailments, of each element of $\Delta$ from $\Gamma$ and vice-versa. For the transitivity of the consequence relations, given $\Gamma_0 = \ax \cup \bigcup_{i \in I} \left(\bigcup_{j_i \in J_i} \Sigma_{j_i} \Phi_{j_i}\right)$, there exists also a $\vdash$-proof of $\Delta$ from $\Gamma_0$. Let $\delta \in \Delta$ and $\Gamma_0 \vdash \Gamma_1 \vdash \ldots \vdash \Gamma_n \supseteq \{\delta\}$ be a proof of $\delta$ from $\Gamma_0$  --- so, each element of $\Gamma_{a+1}$ is directly derivable from $\Gamma_a$ --- such that, for all $a \in \{0, \ldots, n-1\}$, $\Gamma_a \subseteq \Gamma_{a+1}$.\footnote{This assumption does not affect the argument, since we can always substitute a set in the sequence by its union with the previous set.} We shall now show by induction that such a proof can be interpreted in $\vdash'$ via the translation $t$.

\begin{itemize}
\item Induction basis: \emph{For every element $\alpha$ of $\Gamma_0$, there exist $i \in I$, $\phi \in D_i$, and $\sigma \in \sfm$ such that $\alpha = \sigma(\phi)$ and $t(\sigma)(f_i(\phi))$ is directly derivable from $g'((m_i)_{i \in I})$ in $\vdash'$.}

Indeed, for all $\alpha \in \Gamma_0$, there exists $i \in I$ such that either $\alpha = \phi \in \ax_i$ or $\alpha = \sigma(\phi)$ with $\sigma \in \Sigma_{j_i}$ and $\phi \in \Phi_{j_i} \subseteq D_i$ for some $j_i \in J_i$. In the former case, $\vdash' f_i(\alpha)=t(\sigma)(f_i(\phi))$, where $\sigma = \id$, while in the latter $t(\sigma)(f_i(\phi)) \in g'((m_i)_{i \in I}$. In both cases, $t(\sigma)(f_i(\phi))$ is directly derivable from $g'((m_i)_{i \in I})$.

\item Induction step: \emph{For all $a \in \{0, \ldots, n-1\}$ and $\alpha \in \Gamma_{a+1}$, there exists $i \in I$, $\Sigma \in \Sfm$, $\Theta \subseteq \ax_i \cup (\Gamma_a \cap \wp D_i)$ and $\phi \in D_i$ such that $\alpha \in \Sigma \cdot \{\phi\}$ and $t(\Sigma)\cdot f_i(\Theta) \vdash' t(\Sigma) \cdot f_i(\phi)$.}

In other words, we will now prove that every ``global'' deduction, that is, every deduction in $\vdash$, can be traced back to a ``local'' one in some $\vdash_i$, and therefore interpreted in $\vdash'$ by means of $t$ and $f_i$.

Each element $\alpha$ of $\Gamma_{a+1}$ is a consequence of $\Gamma_a$. Therefore, there exists $i \in I$ such that either $\alpha = \phi$ is an axiom of $\vdash_i$ or there exist $\Theta \subseteq \Gamma_a \cap \wp D_i$, $\Lambda \subseteq D_i$, and $\Sigma \in \Sfm$ such that $\frac{\Theta}{\Lambda}$ is an inference rule of $\vdash_i$ and $\alpha = \sigma(\phi) \in \Sigma \cdot \Lambda$. In both cases, $t(\Sigma) \cdot f_i(\Theta) \vdash' t(\sigma)(f_i(\phi))$.
\end{itemize}
Since
$$\begin{array}{l}
g'((m_i)_{i \in I}) = \bigvee_{i \in I} \left(\bigvee_{j_i \in J_i} t(\Sigma_{j_i}) \cdot f_i(\Phi_{j_i})\right), \text{ and} \\
g'((n_i)_{i \in I}) = \bigvee_{i \in I} \left(\bigvee_{k_i \in K_i} t(\Xi_{k_i}) \cdot f_i(\Psi_{j_i})\right),
\end{array}$$
from the steps above it follows $g'((n_i)_{i \in I}) \subseteq g'((m_i)_{i \in I})$. The converse inclusion is completely analogous, hence we have $g'((m_i)_{i \in I}) = g'((n_i)_{i \in I})$ and, therefore, $((m_i)_{i \in I},(n_i)_{i \in I}) \in \ker g'$.

Last, from $\ker g \subseteq \ker g'$, we get that there exists an $\Sfm$-module morphism $f: \th \to \th_t'$ which makes the whole diagram in (\ref{coprdiag}) commutative. This completes the proof.
\end{proof}

In Theorems \ref{qmcoprod} and \ref{ds0copr} we described the coproducts in $\QM$ and $\DS_0$ respectively. In the next example, we will show that the coproduct of a family of objects of $\DS_0$ in $\QM$ need not be isomorphic to the coproduct of the same family in $\DS_0$.
\begin{exm}\label{exmcopr}
Let us consider the language $\lang = (\{\to,\neg\},(2,1))$ and let us consider the deductive system of Classical Propositional Logic on $\lang$. Now consider the coproducts in $\QM$ and $\DS_0$ of two copies of such a system. In order to make the disjoint union of two copies of the language and, therefore, the coproduct of two copies of the quantale of substitutions, let us denote them by $\lang_1 = (\{\to_1,\neg_1\},(2,1))$ and $\lang_2 = (\{\to_2,\neg_2\},(2,1))$. Consequently, the two systems shall be denoted by $({\Sfm}_i,\th_i)$, $i =1,2$.

Theorem \ref{qmcoprod} guarantees that the coproduct of the given pair of systems in $\QM$ is $(\Sfm,M)$, where $M = \coprod_{i=1}^2 \Sfm \tensor_{{\Sfm}_i} \th_i$, and let us denote by $(\Sfm, \th)$ the coproduct of the two systems in $\DS_0$, and by $\gamma$ the nucleus associated to its consequence relation. By \cite[Theorem 5.3]{ruslu} and Theorem \ref{qmcoprod}, there exist ${\Sfm}_i$-module embeddings $e_i: \th_i \to \th$ and $g_i: \th_i \to M$, and a unique $\Sfm$-module morphism $g: M \to \th$ such that $g\circ g_i = e_i$, for $i =1, 2$, as in the diagram (\ref{coprdiag}). We will show that such a morphism $g$ is surjective but not injective. Surjectivity is actually an immediate consequence of \cite[Theorem 5.4]{ruslu} (see also diagram (7) of the same paper).

By \cite[Theorem 5.4]{ruslu}, for all $i \in I$, $\Sfm \tensor_{{\Sfm}_i} \th_i$ is isomorphic to the lattice of theories $\th_i'$ of the deductive system defined on $E$ by means of the axioms and rules of $\vdash_i$. Let us denote by $\delta_i$ the nuclei such that $\th_i' = \wp E_{\delta_i}$, for $i =1,2$, and refer to the following commutative diagram, where $\alpha$ is the isomorphism as in \cite[Theorem 5.4]{ruslu}, and $g'\circ \alpha = g$.
$$\begin{tikzcd}
                                                                  &  & \th                                                                      &  &                                                                  \\
                                                                  &  & \th_1'\times\th_2' \arrow[d, "\alpha^{-1}", shift left] \arrow[u, "g'"'] &  &                                                                  \\
\th_1 \arrow[rr, "g_1"'] \arrow[rru, "g_1'"'] \arrow[rruu, "e_1"] &  & M \arrow[u, "\alpha", shift left]                                        &  & \th_2 \arrow[ll, "g_2"] \arrow[llu, "g_2'"] \arrow[lluu, "e_2"']
\end{tikzcd}$$
For $j \neq i$, and for any variable $x$, the formula $x \to_j x$ is a tautology of the $j$-th system and not one of the $i$-th system, therefore, $x \to_j x \in \delta_j(\varnothing)$ and $x \to_j x \notin \delta_i(\varnothing)$. It follows that $(\delta_1(\varnothing),\delta_2(\varnothing)) \neq (\delta_1(\{x \to_2 x\}), \delta_2(\varnothing))$. On the other hand, $x \to_1 x$ and $x \to_2 x$ are both tautologies in $\th$, hence $g'(\delta_1(\varnothing),\delta_2(\varnothing)) = g'(\delta_1(\{x \to_2 x\}),\delta_2(\varnothing)) = \gamma(\varnothing)$. It follows that $g(\delta_1(\varnothing),\delta_2(\varnothing)) = g(\delta_1(\{x \to_2 x\}),\delta_2(\varnothing)) = \gamma(\varnothing)$ and, therefore, $g$ is not injective.
\end{exm}

Regarding the so-called ``logical amalgamation" presented in \cite[Section 6]{ruslu}, the situation is slightly different. In order to discuss this point in detail, let us first introduce the pertinent notations. 

Let $\{\lang_i\}_{i=1}^2$ be propositional languages, and $S_i = (D_i, \vdash_i)$ be deductive systems, of the same fixed type, over $\lang_i$, with associated nuclei $\g_i$ and modules of theories $\th_i = (\wp D_i)_{\g_i}$, for $i=1,2$. Let also $\cat M$ be a common fragment of the $\lang_i$'s and $S = (C, \vdash_\beta)$ a deductive system on $\cat M$ whose domain is, again, of the same type of $D_1$ and $D_2$, with associated nucleus $\beta$ and module of theories $\th' = \wp C_\beta$, and such that there exist conservative interpretations $r_i: \th' \to \th_i$, $i=1,2$.

Last, with an abuse of notation, namely, identifying the languages with their sets of connectives, let 
$$\lang = ((\lang_1 \setminus \cat M) \sqcup (\lang_2 \setminus \cat M)) \cup \cat M,$$
$\sqcup$ denoting the disjoint union, so that $\cat M$ is now a fragment of $\lang$ too, and $\lang$ has no further overlapping with (the copies of) $\lang_1$ and $\lang_2$.


From \cite[Corollary 4.3]{ruslu} we know that $\Sfm$ is the amalgamated coproduct of $\Sfa$ and $\Sfb$ with respect to the common subquantale $\wp\Sigma_{\cat M}$. In a more general setting, as the one considered in \cite[Section 6]{ruslu}, in which $\cat M$ is not exactly a common fragment of the $\lang_i$'s but has a translation to each of them, we cannot guarantee that the V-formation of quantales of substitutions can be amalgamated, because the category of quantales does not have the amalgamation property. In fact, we cannot even guarantee the existence of all pushouts in $\DS_0$. Therefore, the construction of \cite[Section 6]{ruslu} is useful in such a general setting and in any concrete situation in which one may not want to identify connectives of the initial languages.

Corollary 4.3 of \cite{ruslu} and Theorem \ref{amalgth} provide the correct way to obtain the amalgamated coproduct in $\QM$ of the V-formation
$$(\Sfa,\th_1) \stackrel{(\iota_1,r_1)}{\longleftarrow} (\wp\Sigma_{\cat M},\th') \stackrel{(\iota_2,r_2)}{\longrightarrow} (\Sfb,\th_2),$$
where $\iota_1$ and $\iota_2$ are the quantale embeddings induced by the inclusions of $\cat M$ in the other two languages.

On the other hand, thanks to Theorem \ref{ds0copr}, in this specific case we shall be able to obtain an even more elegant amalgamation in $\DS_0$ than the one presented in \cite{ruslu}. Indeed, with the given notations, let us also denote by $e_i$, $i = 1,2$, the inclusion maps of $\Sfa$ and $\Sfb$ in $\Sfm$, and by $\iota_i$, $i=1,2$, the inclusion maps of $\wp\Sigma_{\cat M}$ in $\Sfa$ and $\Sfb$. Then we have
\begin{theorem}\label{ds0amalg}
The amalgamated coproduct of the V-formation
$$(\Sfa,\th_1) \stackrel{(\iota_1,r_1)}{\longleftarrow} (\wp\Sigma_{\cat M},\th') \stackrel{(\iota_2,r_2)}{\longrightarrow} (\Sfb,\th_2),$$
in $\DS_0$ is $(\Sfm, \th)$, where $\th$ is the module of theories of the $\lang$-system  whose consequence relation is determined by the union of the axioms and rules of $S_1$, $S_2$, and the set of rules
$$\Theta = \left\{\frac{e_ir_i(\{\phi\})}{e_kr_k(\{\phi\})} \bigg| \phi \in C, i \neq k \in \{1,2\}\right\}.$$
\end{theorem}
\begin{proof}
By \cite[Corollary 4.3]{ruslu}, 
$$
\begin{tikzcd}
                                       & \Sfa \arrow[rd, "e_1"]  &   \\
\wp\Sigma_{\cat M} \arrow[ru, "\iota_1"] \arrow[rd, "\iota_2"'] &                        & \Sfm \\
                                       & \Sfb \arrow[ru, "e_2"'] &  
\end{tikzcd}$$
is the diagram of an amalgamated coproduct of quantales. Then the assertion follows readily from Theorem \ref{amalgth} and \cite[Theorems 6.4, 6.5, and 6.6]{ruslu}.
\end{proof}

\bibliographystyle{plain}

\bibliography{C:\\Users\\ciror\\Dropbox\\Ciro\\ricerca\\cirobibtex}

\begin{thebibliography}{10}

\bibitem{arnmarpin18}
P.~Arndt, D.~C. Pinto, and H.~L. Mariano.
\newblock {F}initary {F}ilter {P}airs and {P}ropositional {L}ogics.
\newblock {\em South American Journal of Logic}, 4(2):257--280, 2018.

\bibitem{blojon}
W.~J. Blok and B.~Jónsson.
\newblock Equivalence of consequence operations.
\newblock {\em Studia Logica}, 83(1-3):91--110, 2006.

\bibitem{blpi}
W.~J. Blok and D.~Pigozzi.
\newblock Algebraizable {L}ogics.
\newblock {\em Memoirs of the {A}merican {M}athematical {S}ociety},
  77(396):1--78, 1989.

\bibitem{cinmor}
P.~Cintula, J.~Gil-Férez, T.~Moraschini, and F.~Paoli.
\newblock An abstract approach to consequence relations.
\newblock {\em The Review of Symbolic Logic}, 12(2):331--371, 2019.

\bibitem{gab96}
D.~Gabbay.
\newblock An overview of fibred semantics and the combination of logics.
\newblock In F.~Baader and K.~Schulz, editors, {\em Frontiers of Combining
  Systems}, pages 1--55. Kluwer Academic Publishers, 1996.

\bibitem{galtsi}
N.~Galatos and C.~Tsinakis.
\newblock Equivalence of consequence relations: an order-theoretic and
  categorical perspective.
\newblock {\em Journal of Symbolic Logic}, 74(3):780--810, 2009.

\bibitem{gli}
V.~Glivenko.
\newblock Sur quelques points de la logique de {M}. {B}rouwer.
\newblock {\em Acad. {R}oy. {B}elgique, {B}ull. {C}lasse {S}ci.},
  5(15):183--188, 1929.

\bibitem{krupas}
D.~Kruml and J.~Paseka.
\newblock Algebraic and categorical aspects of quantales.
\newblock In M.~Hazewinkel, editor, {\em Handbook of Algebra, Vol. 5}.
  Elsevier, 2008.

\bibitem{liang}
S.~Liang.
\newblock The coproduct of unital quantales.
\newblock {\em Progress in Applied Mathematics}, 1(2):71--76, 2011.

\bibitem{marpin17}
D.~C. Pinto and H.~L. Mariano.
\newblock Algebraizable logics and a functorial encoding of its morphisms.
\newblock {\em Logic Journal of the IGPL}, 25(4):524--561, 2017.

\bibitem{rosenthal}
K.~I. Rosenthal.
\newblock {\em Quantales and their applications}.
\newblock Longman Scientific and Technical, 1990.

\bibitem{rusthesis}
C.~Russo.
\newblock {\em Quantale Modules, with Applications to Logic and Image
  Processing}.
\newblock PhD thesis, University of Salerno -- Italy, 2007.

\bibitem{rusjlc}
C.~Russo.
\newblock Quantale modules and their operators, with applications.
\newblock {\em Journal of Logic and Computation}, 20(4):917--946, 2010.

\bibitem{rusapal}
C.~Russo.
\newblock An order-theoretic analysis of interpretations among propositional
  deductive systems.
\newblock {\em Annals of Pure and Applied Logic}, 164(2):112--130, 2013.

\bibitem{ruscorrapal}
C.~Russo.
\newblock Corrigendum to ``{A}n order-theoretic analysis of interpretations
  among propositional deductive systems'' [{A}nn. {P}ure {A}ppl. {L}ogic 164
  (2) (2013) 112–130].
\newblock {\em Annals of Pure and Applied Logic}, 167(3):392--394, 2016.

\bibitem{russajl}
C.~Russo.
\newblock Quantales and their modules: projective objects, ideals, and
  congruences.
\newblock {\em South American Journal of Logic}, 2(2):405--424, 2016.

\bibitem{ruslu}
C.~Russo.
\newblock Coproduct and amalgamation of deductive systems by means of ordered
  algebras.
\newblock {\em Logica Universalis}, 16(1-2):355--380, 2022.

\bibitem{ser1}
A.~Sernadas, C.~Sernadas, and C.~Caleiro.
\newblock Fibring of logics as a categorial construction.
\newblock {\em Journal of Logic and Computation}, 9(2):149--179, 1999.

\bibitem{sol}
S.~A. Solovyov.
\newblock On the category {$Q$}-mod.
\newblock {\em Algebra Universalis}, 58:35--58, 2008.

\bibitem{woj}
R.~Wójcicki.
\newblock {\em Theory of Logical Calculi -- {B}asic Theory of Consequence
  Operations}.
\newblock Kluwer Academic Publishers, Dordrecht, 1988.

\end{thebibliography}

\end{document}